\documentclass[12pt]{amsart}
\usepackage[utf8]{inputenc}
\usepackage[T1]{fontenc}
\usepackage[T1]{fontenc}

\usepackage[leqno]{amsmath}
\usepackage{amssymb}
\usepackage{accents}
\usepackage{graphicx}
\usepackage{mathrsfs}
\usepackage{comment, float}
\usepackage[dvipsnames]{xcolor}
\usepackage[foot]{amsaddr}
\usepackage[shortlabels]{enumitem}
\setlist[itemize]{leftmargin=11mm}
\setlist[enumerate]{leftmargin=11mm}
\usepackage{mathdots}
\usepackage[a4paper, footskip=0.5in,  headheight = 0.5in, top=1.25in, bottom=1.25in,  right=1in,  left=1in]{geometry}
\usepackage{hyperref}
\hypersetup{
colorlinks,
linkcolor=black,
urlcolor=Blue,
citecolor=black,}
\usepackage[capitalise,nameinlink]{cleveref}
\usepackage[center]{caption}
\usepackage{eufrak}
\usepackage{marvosym}     
\usepackage{latexsym}
\usepackage[nodayofweek]{datetime}
\usepackage{tikz}
\usepackage{subfig}
\usepackage{thm-restate}
\usepackage{verbatim}
\usepackage[none]{hyphenat} 
\usepackage{cite}

\newtheorem{statement}{Statement}[section]
\newtheorem{theorem}[statement]{Theorem}
\newtheorem{lemma}[statement]{Lemma}

     \DeclareMathOperator{\clx}{<_{{\text{\sf{clex}}}}}
 \DeclareMathOperator{\clxe}{\leq_{{\text{\sf{clex}}}}}
 
 \usepackage{stmaryrd}
\newcommand{\yl}{\hspace{-0.3mm}\mathrel{\scalebox{1.25}{$\Yleft$}}\hspace{-0.3mm}}

\newcommand{\mca}{\mathcal}
\newcommand{\poi}{\mathbb{N}} 
\newcommand{\ol}{\overline}
\newcommand{\eps}{\varepsilon}
\newcommand{\set}[2]{\{#1,\ldots, #2\}}

\newcounter{tbox}
\newcommand{\sta}[1]{\medskip\medskip\refstepcounter{tbox}\noindent{\parbox{\textwidth}{(\thetbox) \emph{#1}}}\vspace*{0.3cm}}
\newcommand{\mylongtitle}[1]{%
  \ifodd\value{page}%
    \protect\parbox{0.97\linewidth}{#1}\hfill%
  \else%
    \hfill\protect\parbox{0.97\linewidth}{#1}%
  \fi%
}

\title{Forcing monochromatic induced subgraphs}

\author{Sepehr Hajebi$^{\dagger}$}

\author{Sophie Spirkl$^{\dagger \ast}$}
\thanks{$^{\dagger}$ Department of Combinatorics and Optimization, University of Waterloo, Waterloo, Ontario, Canada.}
\thanks{$^{\ast}$ We acknowledge the support of the Natural Sciences and Engineering Research Council of Canada (NSERC), [funding reference number RGPIN-2020-03912].
Cette recherche a \'et\'e financ\'ee par le Conseil de recherches en sciences naturelles et en g\'enie du Canada (CRSNG), [num\'ero de r\'ef\'erence RGPIN-2020-03912]. This project was funded in part by the Government of Ontario. This research was conducted while Spirkl was an Alfred P. Sloan Fellow. This research was undertaken, in part, thanks to funding from the Canada Research Chairs Program.}
\date{\today}

\begin{document}

\maketitle
\begin{abstract}
We prove that for all $c\in\mathbb N$ and nonnull graphs $H_1,\ldots,H_t$, there exists $n\in\mathbb N$ such that if $G$ is a $c$-edge-colored complete graph with no monochromatic induced copy of the complete join of $H_1,\ldots,H_t$, then $V(G)$ is the union of $n$ sets $V_1,\ldots,V_n$ such that within each set $V_j$ with $|V_j|\neq 1$, the edges of some color form a graph that excludes at least one of $H_1,\ldots,H_t$ as an induced subgraph. In fact, the same holds even if the colors overlap, and with a different list of graphs $H_1,\ldots,H_t$ assigned to each color. When $H_1,\ldots,H_t$ each have a single vertex, this is Ramsey's theorem, and when $c=2$, this is the ``excluding pairs of graphs'' theorem of Chudnovsky, Scott, and Seymour.
\end{abstract}

\section{Introduction}

In this paper, $\poi$ is the set of all positive integers. A graph $G$ consists of a finite vertex set $V(G)$ and an edge set $E(G)$ with no ``loops'' or ``parallel edges''. For $X\subseteq V(G)$, $G[X]$ is the subgraph of $G$ induced by $X$. A \textit{copy} of a graph $H$ in $G$ is an {\sl induced} subgraph of $G$ isomorphic to $H$, and we say that $G$ is \textit{$H$-free} if there is no copy of $H$ in $G$. 

Ramsey's classical theorem \cite{multiramsey} asserts that sufficiently large  edge-colored complete graphs contain large monochromatic complete subgraphs. (A proof for $c^{c(t-1)}$ is included in the appendix. The best known bound improves this by a factor of $e^{O_{c}(t)}$ \cite{bestmultiramsey}.)

\begin{theorem}[Ramsey \cite{multiramsey}]\label{thm:multiramsey}
For all $c,t\in \mathbb N$ with $c\geq 2$ and every integer $n\geq c^{c(t-1)}$, every $c$-coloring of $E(K_n)$ contains a monochromatic copy of $K_t$.
\end{theorem}

For $c=2$, this is equivalent to the following (where $\ol{G}$ denotes the complement of $G$):

\begin{theorem}[Ramsey \cite{multiramsey}]\label{thm:graphramsey}
For all $t\in \mathbb N$, if $G$ is a graph such that both $G$ and $\ol{G}$ are $K_t$-free, then $|V(G)|< 4^{t-1}$.
\end{theorem}

Another now-classical result, the ``excluding pairs of graphs'' theorem of Chudnovsky, Scott, and Seymour \cite{expairs}, extends \Cref{thm:graphramsey} to excluding copies of arbitrary ``complete joins'' in both $G$ and $\ol{G}$, showing that in this case $V(G)$, although perhaps not small, is the union of a small number of sets each inducing a ``simple'' subgraph of $G$. Recall that the \textit{complete join} of graphs $H_1,\ldots, H_t$, where $t\in \mathbb N$, is the graph $H$ whose vertex set is the union of $t$ pairwise disjoint sets $X_1,\ldots, X_t$ such that for every $i\in \{1,\ldots,t\}$, $H[X_i]$ is isomorphic to $H_i$, and for all distinct $i,j\in \{1,\ldots,t\}$, every vertex in $X_i$ is adjacent in $H$ to every vertex in $X_j$.

\begin{theorem}[Chudnovsky, Scott, Seymour \cite{expairs}]\label{thm:expairs}
    For all integers $t,t'\geq 2$ and nonnull graphs $H_{1},\ldots,H_{t},H'_{1},\ldots,H'_{t'}$, there exists $n\in \poi$ such that if $G$ is a graph with no copy of the complete join of $H_{1},\ldots,H_{t}$ in $G$ and no copy of the complete join of $H'_{1},\ldots,H'_{t'}$ in $\ol{G}$, then $V(G)=V_1\cup\cdots\cup V_n$ such that for each $j\in \{1,\ldots,n\}$, either $|V_j|=1$, or $G[V_j]$ is $H_{i}$-free for some $i\in \set{1}{t}$, or $\ol{G}[V_j]$ is $H'_{i}$-free for some $i\in \set{1}{t'}$.
\end{theorem}

(For $t=1$ or $t'=1$, this is trivial with $n=1$ since the assumption is stronger than the conclusion. Also, \Cref{thm:graphramsey} corresponds to the case where $H_{1},\ldots,H_{t},H'_{1},\ldots,H'_{t'}$ are all singletons, which means $|V_1|,\ldots, |V_n|\leq 1$.)
\medskip

We generalize \Cref{thm:expairs} to edge-colored complete graphs with arbitrarily many colors -- just as \Cref{thm:multiramsey} extends \Cref{thm:graphramsey} -- and in fact in a couple of other ways: We allow the colors to overlap, and instead of a list of single graphs, we exclude complete joins of graphs that form a ``transversal'' of a prescribed list of sets of graphs. 

Let $V$ be a finite set. We denote by $K_V$ the complete graph with vertex set $V$. For $c\geq 2$, a \textit{$c$-multicoloring of $E(K_V)$} is a $c$-tuple $(G_1,\ldots,G_c)$ of graphs with vertex set $V$ and $E(G_1)\cup\cdots\cup E(G_c)=E(K_V)$; this should be thought of as a $c$-coloring of $E(K_V)$ with the edges receiving {\sl at least} one color. A \textit{class} is a set of graphs taken up to isomorphism. For a list $\mathscr{L}=(\mca{H}_1,\ldots,\mca{H}_t)$ of $t\in \poi$ nonempty classes of nonnull graphs, an \textit{$\mathscr{L}$-complete graph} is the complete join of graphs $H_1,\ldots,H_t$ for some $(H_1,\ldots, H_t)\in \mca{H}_1\times \cdots \times \mca{H}_t$. For a class $\mca{H}$, we say that a graph $G$ is \textit{$\mca{H}$-free} if $G$ is $H$-free for all $H\in \mca{H}$.
\medskip

Our main result is the following. (Note that the case $h=1$ contains \Cref{thm:multiramsey} with a matching bound, as $\tau=1$. For $t=1$, this again is trivial with $n=1$.)

\begin{restatable}{theorem}{main}\label{thm:main}
 Let $c,h,t\in \poi$ with $c,t\geq 2$, and let $n=c^{\tau c(t-1)}-1$, where
$$\tau=2^{c^{3cht}-c^{3ct}}.$$
 Let $(G_1,\ldots,G_c)$ be a $c$-multicoloring of $E(K_V)$ for a finite set $V$. For each $i\in \{1,\ldots,c\}$, let $\mathscr{L}_i$ be a nonempty list of at most $t$ nonempty classes of nonnull graphs on at most $h$ vertices, such that $G_i$ has no $\mathscr{L}_i$-complete induced subgraph. Then $V$ is the union of $n$ sets $V_1,\ldots,V_n$ such that for every $j\in \{1,\ldots,n\}$, either $|V_j|=1$ or there exists $i\in \{1,\ldots,c\}$ such that $G_i[V_j]$ is $\mca{H}$-free for some class $\mca{H}$ from the list $\mathscr{L}_i$.
\end{restatable}

\subsection{An application} There is a well-known Ramsey-type result asserting that if a graph $G$ is $\{K_{a},K_{b,b}\}$-free for some $a,b\in \poi$, then every collection of nonempty bounded-size subsets of $V(G)$ either has a small hitting set or contains many pairwise anticomplete sets; where $S,S'\subseteq V(G)$ are \textit{anticomplete} if $S\cap S'=\varnothing$ and there is no edge of $G$ with one end in $S$ and the other in $S'$. (The proof is a straightforward application of \Cref{thm:graphramsey}. See \cite{lozin,pinned}, and also \cite{tw19} for improved bounds.) 

The present work grew out of an earlier attempt to prove an ``$\alpha$-strengthening'' of this result, where the ``$K_a$-free'' assumption is dropped and the ``small hitting set'' conclusion is relaxed to a ``hitting set $X$ such that $\alpha(G[X])$ is bounded'' -- Note that this is stronger: if $G$ also excludes a complete subgraph, then by \Cref{thm:graphramsey}, $X$ has bounded size. (Recall that $\alpha(G)$ denotes the cardinality of a largest stable set in a graph $G$, where a \textit{stable set} is a set of pairwise nonadjacent vertices.) We established this in \cite[Lemma~6.2]{tafinite}, which served as the key ingredient in the resolution of two conjectures from \cite{DKKMMSW} on the tree independence number of graphs in hereditary classes (see \cite{DKKMMSW,DMS2,tafinite} for further discussion). However, the bound on $\alpha(G[X])$ was astronomically large (at least a tower of height equal to the upper bound on the size of the sets), stemming from nested applications of (a special case of) the Graham-Rothschild theorem \cite{GRgeneral}.

Using \Cref{thm:main}, we can now obtain a much better bound:

\begin{restatable}{theorem}{betteranti}\label{thm:better-hit-vs-anti}
   Let $b,c,d\in \mathbb N$ with $b,c\geq 2$, and let $n=c^{c\tau(t-1)}-1$, where
   $$t=2^{2^{2c^2+1}c^2 (2b+d-1)};$$
   $$\tau=2^{c^{3bct}-c^{3ct}}.$$
   Let $G$ be a $K_{b,b}$-free graph and let $\mca{S}$ be a set of subsets of $V(G)$ such that $0<|S|<c$ for every $S\in \mca{S}$. Then either $\mca{S}$ admits a hitting set $X\subseteq V(G)$ such that $\alpha(G[X])\leq bcn$, or there are $d$ sets $S_1,\ldots,S_d$ in $\mca{S}$ that are pairwise anticomplete in $G$.
\end{restatable}

The proof of \Cref{thm:better-hit-vs-anti} closely follows our ideas from \cite[Lemma~6.2]{tafinite}, and so we include it in the appendix. As the reader may notice, somewhat interestingly, the overlapping color classes are crucial to the application of \Cref{thm:main} in that proof.

\section{The main proof}
For the proof of \Cref{thm:main}, we need a few lemmas, and before those, several definitions:
\begin{itemize}[leftmargin=8mm, itemsep=1mm]

    \item Fo integers $k,k'$, we denote by $\set{k}{k'}$ the set of all positive integers at least $k$ and at most $k'$ (and $\set{k}{k'}=\varnothing$ if and only if $k'<k$). For $L\subseteq \poi$, we define 
$$L\yl L=\{(i,j)\in L\times L: i<j\}.$$
We denote by $\clx$ the \textit{co-lexicographic order} on $\poi \yl \poi$: For $(i,j),(i',j')\in \poi \yl \poi$, we have $(i,j)\clx(i',j')$ if and only if either $j<j'$ or $j=j'$ and $i<i'$. The first few elements of $\poi \yl \poi$ according to this order are:
$$(1,2)\ \clx\ (1,3)\ \clx\ (2,3)\ \clx\ (1,4)\ \clx\ (2,4)\ \clx\ (3,4)\ \clx\ (1,5)\ \clx \cdots\hspace{-7mm}$$
We also write $(i,j)\clxe (i',j')$ to mean that either $(i,j)=(i',j')$ or $(i,j)\clx(i',j')$. For every $(i,j)\in \poi \yl \poi$, by the \textit{$\clx$-successor of $(i,j)$} we mean the (unique) minimum pair $(i',j')\in \poi \yl \poi$ with respect to $\clx$ such that $(i,j)\clx (i',j')$. (Specifically, $(i',j')=(i+1,j)$ if $i<j-1$, and $(i',j')=(1,j+1)$ if $i=j-1$.)

\item A \textit{bisequence} is a function $\eps:\poi \yl \poi\to (0,1]$. Given a bisequence $\eps$, we adopt the convention that 
$$\prod_{(i,j)\in \varnothing}\eps(i,j)=1.$$
We denote by $\eps^d$, where $d\in \poi\cup \{0\}$, the bisequence with the rule 
$$\eps^d:(i,j)\mapsto (\eps(i,j))^d$$
and by $\pi_{\eps}$ the bisequence with the rule 
$$\pi_{\eps}:(p,q)\mapsto \prod_{(i,j)\clx (p,q)}\eps(i,j).$$
In particular, $\pi_{\eps}(1,2)=1$ for every bisequence $\eps$.

\item By a \textit{space} we mean a pair $(\Omega,\mu)$ where $\Omega$ is a finite nonempty set and $\mu:2^{\Omega}\to [0,1]$ satisfies $\mu(\varnothing)=0$, $\mu(\Omega)=1$, and $\mu(X)\leq \mu(Y)+\mu(Z)$ for all $X,Y,Z\subseteq \Omega$ with $X\subseteq Y\cup Z$. (In particular, $\mu(X)\leq \mu(Y)$ for all $X\subseteq Y\subseteq \Omega$.)

\item For a vertex $x$ of a graph $G$, we denote the set of all neighbors of $x$ in $G$ by $N_G(x)$. Given a space $(\Omega,\mu)$, a graph $G$ with $V(G)=\Omega$, and $\eps\in (0,1]$, we say that a pair $(X,Y)$ of disjoint subsets of $\Omega$ is \textit{$(\mu,\eps)$-sparse in $G$} if $\mu(N_G(x)\cap Y)< \eps\mu(Y)$ for every vertex $x\in X$.

\end{itemize}

We may now prove our first lemma. The argument is similar to one from the work of Erd\H{o}s and Hajnal \cite{EH}, but needs careful adaptation:

\begin{lemma}\label{lem:mainpure} Let $h\in \poi$ and let $\eps$ be a bisequence taking values in $(0,1/h]$. Let $(\Omega,\mu)$ be a space, let $H$ be a nonnull graph on at most $h$ vertices, and let $G$ be an $H$-free graph with $V(G)=\Omega$. Let $L\subseteq \poi\setminus \{1\}$ with $|L|=h$, and let $(S_{\ell}:\ell\in L)$ be pairwise disjoint nonempty subsets of $\Omega$. Then there exist $(p,q)\in L\yl L$, $X\subseteq S_p$, and $Y\subseteq S_q$, such that
$\mu(X)\geq \pi_{\eps}(p,q)\mu(S_p)$, $\mu(Y)\geq \pi_{\eps}(p,q)\mu(S_q)$, and $(X,Y)$ is $(\mu,\eps(p,q))$-sparse in $G$ or $\ol{G}$.
\end{lemma}

\begin{proof}
For all $p,q\in \poi$ with $p\leq q$, define 
$$\tilde{\eps}(p,q)=\prod_{i=1}^{p-1}\eps(i,q)\in (0,1].$$
In particular, $\tilde{\eps}(1,q)=1$ for every $q\in \poi$.

We may assume without loss of generality that $|V(H)|=h$; say,
$V(H)=\{v_{\ell}:\ell\in L\}$.
Since $V(G)=\Omega\neq \varnothing$ and $G$ is $H$-free, it follows that $h=|V(H)|\geq 2$. 

For every $p\in L$, let 
$$I_p=\{i\in L:i<p\}$$
and let 
$$J_p=\{j\in L:j\geq p\}.$$
By a \textit{$p$-chain in $G$} we mean an $h$-tuple $(W_{\ell}:\ell\in L)$ with the following specification:

 \begin{itemize}[leftmargin=8mm, itemsep=1mm]
     \item For each $\ell\in L$, we have $W_{\ell}\subseteq S_{\ell}$;
     \item For each $\ell\in I_p$, we have $|W_{\ell}|=1$; say, $W_{\ell}=\{w_{\ell}\}$, such that for every $\ell'\in L\setminus \{\ell\}$, 
     \begin{itemize}[leftmargin=8mm, itemsep=1mm]
         \item if $v_{\ell}v_{\ell'}\in E(H)$, then $W_{\ell'}\subseteq N_G(w_{\ell})$; and 
         \item if $v_{\ell}v_{\ell'}\notin E(H)$, then $W_{\ell'}\subseteq N_{\ol{G}}(w_{\ell})$.
     \end{itemize}
     \item For each $\ell \in J_p$, 
     we have $\mu(W_{\ell})\geq \tilde{\eps}(p,\ell)\mu(S_{\ell})$.
 \end{itemize}
It follows that $(S_{\ell}:\ell\in L)$ is a $1$-chain in $G$. Choose the maximum $p\in L$ for which there is a $p$-chain $(W_{\ell}:\ell\in L)$ in $G$. 

We claim that:

\sta{\label{st:sparsepair} For every $x\in W_p$, there exists $q_x\in J_p\setminus \{p\}$ such that 
\begin{itemize}[leftmargin=8mm, itemsep=1mm]
    \item if $v_pv_{q_x}\in E(H)$, then $\mu(N_G(x)\cap W_{q_x})< \eps(p,q_x)\mu(W_{q_x})$; and 
    \item if $v_pv_{q_x}\notin E(H)$, then $\mu(N_{\ol{G}}(x)\cap W_{q_x})<\eps(p,q_x)\mu(W_{q_x})$.
\end{itemize}}

Suppose not. Then there exists $w_p\in W_p$ such that for every $\ell\in J_p\setminus\{p\}$, 
\begin{itemize}[leftmargin=8mm, itemsep=1mm]
    \item if $v_pv_{\ell}\in E(H)$, then $\mu(N_G(w_p)\cap W_{\ell})\geq \eps(p,\ell)\mu(W_{\ell})$; and 
    \item if $v_pv_{\ell}\notin E(H)$, then $\mu(N_{\ol{G}}(w_p)\cap W_{\ell})\geq \eps(p,\ell)\mu(W_{\ell})$.
\end{itemize}
Note that if $J_p=\{p\}$ (or equivalently $p=\max L$), then $G[\{w_{\ell}:\ell\in I_p\}\cup \{w_p\}]$ is isomorphic to $H$, a contradiction to the assumption that $G$ is $H$-free. We deduce that $J_p\neq \{p\}$. Let $p^+=\min(J_p\setminus \{p\})$. Then $I_{p^+}=I_p\cup \{p\}$ and $J_{p^+}=J_p\setminus \{p\}$. Define the $h$-tuple $(W'_{\ell}:\ell\in L)$ as follows:
\[W'_{\ell}=\begin{cases}
    \{w_{\ell}\} &\text{if } \ell\in I_{p^+};\\
    N_G(w_p)\cap W_{\ell} &\text{if } \ell\in J_{p^+}\text{ and }v_pv_{\ell}\in E(H);\\
     N_{\ol{G}}(w_p)\cap W_{\ell} &\text{if } \ell\in J_{p^+}\text{ and }v_pv_{\ell}\notin E(H).
\end{cases}\]
We claim that $(W'_{\ell}:\ell\in L)$ is a $p^+$-chain in $G$, which would then contradict the maximality of $p$. It is easy to observe that $(W'_{\ell}:\ell\in L)$ satisfies the first two bullets of the definition. To check the third bullet, for every $\ell\in J_{p^+}=J_p\setminus \{p\}$, since $\ell\geq p^+\geq p+1$, it follows that
$$\eps(p,\ell)\mu(W_{\ell})\geq (\eps(p,\ell)\cdot\tilde{\eps}(p,\ell)) \mu(S_{\ell})=\tilde{\eps}(p+1,\ell)\mu(S_{\ell})\geq \tilde{\eps}(p^+,\ell)\mu(S_{\ell}).$$
Consequently, if $v_pv_{\ell}\in E(H)$, then 
$$\mu(W'_{\ell})=\mu(N_G(w_p)\cap W_{\ell})\geq \eps(p,\ell)\mu(W_{\ell})\geq \tilde{\eps}(p^+,\ell)\mu(S_{\ell});$$
and if $v_pv_{\ell}\notin E(H)$, then
$$\mu(W'_{\ell})=\mu(N_{\ol{G}}(w_p)\cap W_{\ell})\geq \eps(p,\ell)\mu(W_{\ell})\geq \tilde{\eps}(p^+,\ell)\mu(S_{\ell}).$$
This proves \eqref{st:sparsepair}.
\medskip

For each $x\in W_p$, let $q_x\in J_p\setminus \{p\}$ be as given by \eqref{st:sparsepair}. Since $|J_p\setminus \{p\}|<h$, there exist $X\subseteq W_p$ and $q\in J_p\setminus \{p\}$ such that $\mu(X)\geq \mu(W_p)/h$ and $q_x=q$ for all $x\in X$. Let $Y=W_q$. Then $X\subseteq S_p$ and $Y\subseteq S_q$. 

Recall that $1\notin L$, and so $p\geq 2$. Since $q\geq p+1$, we have $(p-1,q)\clx (p,q)$. Therefore, since $\eps(p-1,q)\leq 1/h\leq 1$, it follows that
$$\pi_{\eps}(p,q)\leq \left(\prod_{i=1}^{p-1}\eps(i,p)\right)\cdot \eps(p-1,q)=\tilde{\eps}(p,p)\cdot \eps(p-1,q)\leq \tilde{\eps}(p,p)/h.$$
Combining this with $\mu(W_p)\geq \tilde{\eps}(p,p)\mu(S_p)$ yields
$$\mu(X)\geq \mu(W_p)/h\geq (\tilde{\eps}(p,p)/h)\cdot \mu(S_p)\geq \pi_{\eps}(p,q)\mu(S_p).$$
Note also that $\pi_{\eps}(p,q)\leq \tilde{\eps}(p,q)$, and so
$$\mu(Y)=\mu(W_q)\geq \tilde{\eps}(p,q)\mu(S_q)\geq \pi_{\eps}(p,q)\mu(S_q).$$
Furthermore, since $q_x=q$ for all $x\in X$, it follows from \eqref{st:sparsepair} that if $v_pv_q\in E(H)$, then $(X,Y)$ is $(\mu,\eps(p,q))$-sparse in $G$, and if $v_pv_q\notin E(H)$, then $(X,Y)$ is $(\mu,\eps(p,q))$-sparse in $\ol{G}$. This completes the proof of \Cref{lem:mainpure}.
\end{proof}

For our second lemma, we need another definition: We say that a space $(\Omega,\mu)$ is \textit{$\alpha$-atomic}, where $\alpha\in (0,1]$, if $\mu(\{x\})\leq \alpha$ for every $x\in \Omega$.

\begin{lemma}\label{lem:largesubsets}
Let $n\in \poi$, let $\alpha,\beta\in (0,1]$ with $\alpha+\beta\leq 1/n$, and let $(\Omega, \mu)$ be an $\alpha$-atomic space. Then there are $n$ pairwise disjoint subsets $S_1,\ldots, S_n$ of $\Omega$ such that $\mu(S_1),\ldots,\mu(S_n)\geq \beta$.
\end{lemma}

\begin{proof}
We prove a stronger statement: For each $k\in \set{1}{n}$ and every $\Omega'\subseteq \Omega$ with $\mu(\Omega')\geq k/n$, there are $k$ pairwise disjoint subsets $S_1,\ldots, S_k$ of $\Omega'$ such that $\mu(S_1),\ldots,\mu(S_k)\geq \beta$. The proof is by induction on $k$. The case $k=1$ is immediate, as $\beta\leq 1/n$. For $k\geq 2$, since $\Omega'$ is finite, we may choose a minimal subset $S_k$ of $\Omega'$ with $\mu(S_k)\geq \beta>0$; in particular, $S_k\neq \varnothing$. Let $x\in S_k$. Since $S_k$ is minimal, it follows that $\mu(S_k\setminus \{x\})<\beta$, and so
$$\mu(S_k)\leq \mu(\{x\})+\mu(S_k\setminus \{x\})
\leq \alpha+\beta\leq 1/n.$$
Consequently,
$$\mu(\Omega'\setminus S_k)\geq \mu(\Omega')-\mu(S_k)\geq (k-1)/n.$$
Now, by the inductive hypothesis, there are pairwise disjoint subsets $S_1,\ldots, S_{k-1}$ of $\Omega'\setminus S_k$ with $\mu(S_1),\ldots,\mu(S_{k-1})\geq \beta$, which along with $S_k$ comprise the desired $k$ subsets of $\Omega'$.
\end{proof}

The following captures the bulk of the difficulty in the proof of \Cref{thm:main}:

\begin{lemma}\label{lem:maindense} Let $\eta\in (0,1]$ and let $a,b,c,h,t\in \poi$ with $c\geq 2$, $b\geq \max\{c,h,1/\eta\}$, and
$$a\geq b^{(2c+2)^{\tbinom{c^{ch}}{2}-1}}.$$
Let $\alpha,\beta\in (0,1]$ with $\alpha+\beta a\leq c^{-ch}$. Let $H_1,\ldots, H_c$ be nonnull graphs on at most $h$ vertices, let $(\Omega,\mu)$ be an $\alpha$-atomic space, and let $(G_1,\ldots, G_c)$ be a $c$-multicoloring of $E(K_{\Omega})$ such that $G_i$ is $H_i$-free for every $i\in \set{1}{c}$. Then there exist $k\in \set{1}{c}$ and disjoint $X,Y\subseteq\Omega$ with $\mu(X),\mu(Y)\geq \beta$ such that $(X,Y)$ is $(\mu,\eta)$-sparse in $\ol{G_k}$. 
\end{lemma}

\begin{proof}
Define two bisequences $\delta,\eps$ recursively, as follows: 
\medskip

\begin{itemize}[leftmargin=8mm, itemsep=1mm]
    \item Let $\delta(1,2)=1$ and let
 $\eps(1,2)=1/b<1$.

\item For $(p,q)\neq (1,2)$ with $\delta(i,j)$ and $\eps(i,j)$ defined for all $(i,j)\clx (p,q)$, let
$$\delta(p,q)=\pi_{\delta^c}(p,q)\cdot \pi_{\eps}(p,q)$$
and let
$$\eps(p,q)=\pi_{\delta^c}(p,q)\cdot \delta^c(p,q).$$
\end{itemize}
\medskip

\noindent Note that $\eps(p,q)$ depends on $\delta(p,q)$, which is defined first. Also, using $b\geq \max\{c,h,1/\eta\}$, it is straightforward to check that:

\sta{\label{st:largetheta} For all $(p,q)\in \poi \yl \poi$, we have $\eps(p,q)\leq \min\{\delta^c(p,q)/c, 1/h,  \eta\}$.}

Throughout, let 
$n=c^{ch}$
and let
$$\mca{I}=\set{1}{n}\yl \set{1}{n}=\{(i,j)\in \poi \yl \poi:(i,j)\clx (1,n+1)\}.$$

First, we need a ``robust'' notion of sparsity (similar to ones appearing in several earlier works; see, for instance, \cite{pp1}). Let $X,Y\subseteq \Omega$ be disjoint, let $(p,q)\in \mca{I}$ and let $d\in \{0,\ldots, c\}$. A \textit{$(p,q)$-refinement of $(X,Y)$ of depth $d$} is a $(2d+3)$-tuple 
$(X_0,\ldots, X_d; Y_0,\ldots, Y_d; f)$
such that:

\begin{itemize}[leftmargin=8mm, itemsep=1mm]
    \item $X=X_0\supseteq\cdots\supseteq X_d$ and $Y=Y_0\supseteq\cdots\supseteq Y_d$;
    \item $f: \set{1}{d}\to \set{1}{c}$ is an injection; and
\item for every $i\in \set{1}{d}$, we have 
\begin{itemize}[leftmargin=8mm, itemsep=1mm]
    \item $\mu(X_i)\geq \delta(p,q)\mu(X_{i-1})$;
    \item $\mu(Y_i)\geq \delta(p,q)\mu(Y_{i-1})$; and
    \item $(X_i,Y_i)$ is $(\mu,\eps(p,q))$-sparse in $G_{f(i)}$.
\end{itemize}
\end{itemize}
In particular, $(X;Y;f:\varnothing\to \set{1}{c})$ is the (unique) depth-zero $(p,q)$-refinement of $(X,Y)$. The \textit{$(p,q)$-tenacity of $(X,Y)$} is the maximum depth of a $(p,q)$-refinement of $(X,Y)$. We show that:

\sta{\label{st:notallcolors} For all disjoint $X,Y\subseteq \Omega$ with $\mu(X)>0$ and every $(p,q)\in \mca{I}$, the $(p,q)$-tenacity of $(X,Y)$ is smaller than $c$.}

Suppose that $(X,Y)$ has a $(p,q)$-refinement $(X_0,\ldots, X_c; Y_0,\ldots, Y_c; f)$ of depth $c$. Note that
$\mu(X_c)\geq \delta^c(p,q)\mu(X)>0$; in particular, $X_c\neq\varnothing$. Let $x\in X_c$. For each $i\in \set{1}{c}$, since $x\in X_c\subseteq X_i$, $Y_c\subseteq Y_i$, and $(X_i,Y_i)$ is $(\mu,\eps(p,q))$-sparse in $G_{f(i)}$, it follows that
$$\mu(N_{G_{f(i)}}(x)\cap Y_{c})\leq \mu(N_{G_{f(i)}}(x)\cap Y_{i})<\eps(p,q)\mu(Y_i)\leq \left(\eps(p,q)\cdot \delta^{i-c}(p,q)\right)\mu(Y_c).$$
Also, by \eqref{st:largetheta}, we have $\eps(p,q)\leq \delta^c(p,q)/c\leq \delta^{c-i}(p,q)/c$. Therefore,
$$\mu(N_{G_{f(i)}}(x)\cap Y_{c})<\mu(Y_c)/c.$$
But now since $Y_c=\bigcup_{i=1}^c(N_{G_{i}}(x)\cap Y_{c})$ and $\set{f(1)}{f(c)}=\set{1}{c}$, it follows that
$$\mu(Y_c)\leq \mu\left(\bigcup_{i=1}^c(N_{G_{f(i)}}(x)\cap Y_{c})\right)\leq \sum_{i=1}^{c} \mu(N_{G_{f(i)}}(x)\cap Y_{c})< \sum_{i=1}^{c} \mu(Y_c)/c=\mu(Y_c);$$
a contradiction. This proves \eqref{st:notallcolors}.
\medskip

Since $(\Omega,\mu)$ is $\alpha$-atomic and $\alpha+\beta a\leq c^{-ch}=1/n$, by \Cref{lem:largesubsets}, there are $n$ pairwise disjoint subsets $S_1,\ldots, S_{n}$ of $\Omega$ with
$$\mu(S_1),\ldots, \mu(S_{n})\geq \beta a>0.$$
Let $\mca{I}^+=\mca{I}\cup \{(1,n+1)\}$ (note that $(1,n+1)$ is the $\clx$-successor of $(n-1,n)$, which in turn is the maximum element of $\mca{I}$ with respect to $\clxe$). Define the sets
$$\left(S_{p,q,1},\ldots, S_{p,q,n}:\ (p,q)\in \mca{I}^+\right)$$
in ``backward recursion'' with respect to $\clxe$, as follows: 

\begin{itemize}[leftmargin=8mm, itemsep=1mm]
    \item For $(p,q)=(1,n+1)$, let 
$(S_{1,n+1,1},\ldots, S_{1,n+1,n})=(S_1,\ldots, S_n)$.
\item Let $(p,q)\in \mca{I}$ and assume that $S_{r,s,1},\ldots,S_{r,s,n}$ are defined for the $\clx$-successor $(r,s)\in \mca{I}^+$ of $(p,q)$. Let $d_{p,q}\leq c$ be the $(p,q)$-tenacity of $(S_{r,s,p},S_{r,s,q})$. Choose a $(p,q)$-refinement $(X_0,\ldots, X_{d_{p,q}}; Y_0,\ldots, Y_{d_{p,q}}; f_{p,q})$ of $(S_{r,s,p},S_{r,s,q})$ of depth $d_{p,q}$. Let
    \[S_{p,q,\ell}=\begin{cases}
        X_{d_{p,q}} &\quad \text{if }\ell=p;\\
        Y_{d_{p,q}} &\quad \text{if } \ell=q;\\
        S_{r,s,\ell} &\quad \text{if } \ell\in \set{1}{n}\setminus \{p,q\}.
    \end{cases}\]
\end{itemize}
(Intuitively, this shirks the pairs $(S_i,S_j)$ into sparse pairs one at a time, without severely ``undoing'' the sparse pairs obtained from previous iterations.)
\medskip

It is immediate from the construction that:

\sta{\label{st:remainlarge} Let $\ell\in \set{1}{n}$. Then, for all $(p_1,q_1),(p_2,q_2)\in \mca{I}^+$ with $(p_1,q_1)\clxe (p_2,q_2)$, we have $S_{p_1,q_1,\ell}\subseteq S_{p_2,q_2,\ell}$ and 
$$\mu(S_{p_1,q_1,\ell})\geq 
    \left(\displaystyle\prod_{(p_1,q_1)\clxe (i,j)\clx (p_2,q_2)} \delta^{d_{i,j}}(i,j)\right)\mu(S_{p_2,q_2,\ell}).$$
In particular, for every $(p,q)\in \mca{I}^+$, we have $S_{1,2,\ell}\subseteq S_{p,q,\ell}$ and
   $$\mu(S_{1,2,\ell})\geq 
   \pi_{{\delta}^c}(p,q)\mu(S_{p,q,\ell}).$$}

We further claim that:

\sta{\label{st:missedcolor} There exist $k\in \set{1}{c}$ and an $(h+1)$-subset $L$ of $\set{1}{n}$ such that for every $(p,q)\in L\yl L$, we have $k\notin \{f_{p,q}(1),\ldots, f_{p,q}(d_{p,q})\}$.}

Let $(p,q)\in \mca{I}$ and let $(r,s)$ be the $\clx$-successor of $(p,q)$. Since $\mu(S_{1,n+1,p})=\mu(S_{p})>0$ and since $(r,s)\clxe (1,n+1)$, it follows from \eqref{st:remainlarge} that $\mu(S_{r,s,p})>0$. Thus, by \eqref{st:notallcolors}, we have $d_{p,q}<c$; that is, there exists $k_{p,q}\in \set{1}{c}$ such that $k_{p,q}\notin \{f_{p,q}(1),\ldots, f_{p,q}(d_{p,q})\}$. Since $c\geq 2$ and $n=c^{ch}$, by \Cref{thm:multiramsey}, there exist $k\in \set{1}{c}$ and an $(h+1)$-subset $L$ of $\set{1}{n}$ such that $k_{p,q}=k$ for every $(p,q)\in L\yl L$. This proves \eqref{st:missedcolor}.
\medskip

From now on, let $k\in \set{1}{c}$ and the $(h+1)$-subset $L$ of $\set{1}{n}$ be as given by \eqref{st:missedcolor}. Recall that $G_k$ is $H_k$-free, and by \eqref{st:largetheta}, $\eps(p,q)\leq 1/h$ for all $(p,q)\in \poi \yl \poi$. So we can apply \Cref{lem:mainpure} to $(S_{1,2,\ell}:\ell\in L\setminus \{1\})$ to obtain some $(p,q)\in (L\setminus \{1\})\yl (L\setminus \{1\}) \subseteq \mca{I}$ for which there exist $X\subseteq S_{1,2,p}$ and $Y\subseteq S_{1,2,q}$ such that

\begin{itemize}[leftmargin=8mm, itemsep=1mm]
    \item $\mu(X)\geq \pi_{\eps}(p,q)\mu(S_{1,2,p})$ and $\mu(Y)\geq \pi_{\eps}(p,q)\mu(S_{1,2,q})$; and
     \item $(X,Y)$ is $(\mu,\eps(p,q))$-sparse in either $G_k$ or $\ol{G_{k}}$.
\end{itemize}
\medskip

In fact, we show that:

\sta{\label{st:notsparse}$(X,Y)$ is $(\mu,\eta)$-sparse in $\ol{G_{k}}$.}

Suppose not. By \eqref{st:largetheta}, we have $\eps(p,q)\leq \eta$. So, $(X,Y)$ is not $(\mu,\eps(p,q))$-sparse in $\ol{G_{k}}$, which means $(X,Y)$ must be $(\mu,\eps(p,q))$-sparse in $G_k$. Let $(r,s)\in \mca{I}^+$ be the $\clx$-successor of $(p,q)$. Recall the $(p,q)$-refinement $(X_0,\ldots, X_{d_{p,q}}; Y_0,\ldots, Y_{d_{p,q}}; f_{p,q})$ of $(S_{r,s,p},S_{r,s,q})$ used to define $(S_{p,q,1},\ldots, S_{p,q,n})$. By \eqref{st:remainlarge}, we have
$S_{1,2,\ell}\subseteq S_{p,q,\ell}$ and
$\mu(S_{1,2,\ell})\geq 
   \pi_{{\delta}^c}(p,q)\mu(S_{p,q,\ell})$ for every $\ell\in \{p,q\}$. Therefore,
       $$\mu(X)\geq \pi_{\eps}(p,q)\mu(S_{1,2,p})\geq \left(\pi_{\delta^c}(p,q)\cdot \pi_{\eps}(p,q)\right)\mu(S_{p,q,p})= 
   \delta(p,q)\mu(X_{d_{p,q}})$$
       and 
         $$\mu(Y)\geq \pi_{\eps}(p,q)\mu(S_{1,2,q})\geq \left(\pi_{\delta^c}(p,q)\cdot \pi_{\eps}(p,q)\right)\mu(S_{p,q,q})= 
   \delta(p,q)\mu(Y_{d_{p,q}}).$$
   
\noindent In summary, $(X,Y)$ is a $(\mu,\eps(p,q))$-sparse pair in $G_k$ such that
\begin{itemize}[leftmargin=8mm, itemsep=1mm]
    \item $X\subseteq S_{1,2,p}\subseteq S_{p,q,p}=X_{d_{p,q}}$ with $\mu(X)\geq \delta(p,q)\mu(X_{d_{p,q}})$; and
    \item $Y\subseteq S_{1,2,q}\subseteq S_{p,q,q}=Y_{d_{p,q}}$ with  $\mu(Y)\geq \delta(p,q)\mu(Y_{d_{p,q}})$.
\end{itemize}  
Let $f:\set{1}{d_{p,q}+1}\to \set{1}{c}$ be the function with the rule
$$f|_{\set{1}{d_{p,q}}}=f_{p,q}\quad \text{and}\quad f(d_{p,q}+1)=k.$$
Then $f$ is an injection since $f_{p,q}$ is an injection and, by \eqref{st:missedcolor}, $k\notin \set{f_{p,q}(1)}{f_{p,q}(d_{p,q})}$. But now
$$(X_0,\ldots, X_{d_{p,q}},X; Y_0,\ldots, Y_{d_{p,q}},Y; f)$$
is a $(p,q)$-refinement of $(S_{r,s,p},S_{r,s,q})$ of depth $d_{p,q}+1$, a contradiction since $d_{p,q}$ is the $(p,q)$-tenacity of $(S_{r,s,p},S_{r,s,q})$. This proves \eqref{st:notsparse}.
\medskip

We also need the following:

\sta{\label{st:m>} $a\geq 1/\delta(1,n+1)$.}
For each $i\in \poi$, let $e_i=\log_{1/b}\delta(p,q)\geq 0$ and let $e'_i=\log_{1/b}\eps(p,q)\geq 0$, where $(p,q)$ is the $i$-th smallest element of $\poi \yl \poi$ with respect to $\clx$. We prove that
$e_i\leq (2c+2)^{i-2}$ for every $i\geq 2$; note that \eqref{st:m>} follows from this for $i=\binom{n}{2}+1$. From the recursions for $\delta,\eps$, it follows that for every $i\geq 2$, we have
$$e_i=\sum_{j=1}^{i-1}ce_j+\sum_{j=1}^{i-1}e'_j;$$
$$e'_i=\sum_{j=1}^{i}ce_j.$$
Note that $e_2=1$ because 
$$\delta(1,3)=\delta^c(1,2)\cdot \eps(1,2)=1/b.$$
Assume, inductively, that $e_i\leq (2c+2)^{i-2}$ for some $i\geq 2$. Then
\begin{align*}
 e_{i+1}&=\sum_{j=1}^{i}ce_j+\sum_{j=1}^{i}e'_j\\
 &=ce_i+\left(\sum_{j=1}^{i-1}ce_j+\sum_{j=1}^{i-1}e'_j\right)+e'_i\\
 &=(c+1)e_i+e'_i\\
 &=(2c+1)e_i+\sum_{j=1}^{i-1}ce_j\\
 &\leq (2c+1)e_i+\sum_{j=1}^{i-1}ce_j+\sum_{j=1}^{i-1}e'_j\\
 &=(2c+2)e_i.
\end{align*}
This proves \eqref{st:m>}.
\medskip

Now, by \eqref{st:remainlarge}, we have 
     $\mu(S_{1,2,p})\geq 
     \pi_{\delta^c}(1,n+1)\mu(S_{p})$ and $\mu(S_{1,2,q})\geq 
     \pi_{\delta^c}(1,n+1)\mu(S_{q})$, and by \eqref{st:m>}, we have $$\mu(S_{p}),\mu(S_{q})\geq \beta a\geq \beta/\delta(1,n+1).$$ Since $\pi_{\eps}(p,q)\geq \pi_{\eps}(1,n+1)$, it follows that
       $$\mu(X)\geq \pi_{\eps}(p,q)\mu(S_{1,2,p})\geq \left(\pi_{\delta^c}(1,n+1)\cdot \pi_{\eps}(1,n+1)\right)\mu(S_{p})= 
       \delta(1,n+1)\mu(S_{p})\geq \beta$$
       and likewise
       $$\mu(Y)\geq \pi_{\eps}(p,q)\mu(S_{1,2,q})\geq \left(\pi_{\delta^c}(1,n+1)\cdot \pi_{\eps}(1,n+1)\right)\mu(S_{q})= 
       \delta(1,n+1)\mu(S_{q})\geq \beta.$$
Also, by \eqref{st:notsparse}, $(X,Y)$ is $(\mu,\eta)$-sparse in $\ol{G_{k}}$. This completes the proof of \Cref{lem:maindense}.
\end{proof}

We may now prove our main result:

\main*
\begin{proof} Recall that if $h=1$, then $\tau=1$ and this follows from \Cref{thm:multiramsey} (with a minor adjustment to handle the overlapping colors, which we leave to the reader). Assume that $h\geq 2$ (and recall that $c,t\geq 2$). Let 
$$b=c^{2^{ht}};$$
let
$$a=b^{(2c+2)^{\tbinom{c^{cht}}{2}}};$$
and let 
$$\lambda=3ac^{cht}.$$

We show that:

\sta{\label{st:lambdabound} $b\geq \max\{c,2ht\}$ and $\lambda< c^{\tau}$.}
Clearly $b\geq c$, and since $c\geq 2$, it follows that $b=c^{2^{ht}}>2^{ht}=2\cdot 2^{ht-1}\geq 2ht$. It remains to show that $\lambda< c^{\tau}$. Note that $b= c^{2^{ht}}>c^{ht}>3$ because $c,t\geq 2$. It follows that
$$\lambda=3ac^{cht}<b^{(2c+2)^{\tbinom{c^{cht}}{2}}+c+1}.$$
Since $c\geq 2$, we have $(c+1)^2\geq 3c+3$, and so
$$\log_b\lambda< (2c+2)^{\tbinom{c^{cht}}{2}}+c+1<(3c+3)^{\tbinom{c^{cht}}{2}}\leq (c+1)^{c^{2cht}-c^{cht}}<2^{c^{3cht}-c^{2cht}};$$
where the last inequality holds because $c+1\leq 2^c<2^{cht}$. We deduce that
$$\log_c\lambda=\log_b\lambda \cdot \log_c b< 2^{c^{3cht}-c^{2cht}+ht}.$$
On the other hand, since $c\geq 2$, it follows that $c^{3ct}>2^t>t$, and from $c,h,t\geq 2$, it follows that $c^{(2h-3)ct}>2^{8h-12}>8h-12>h+1$. Therefore,
$$c^{2cht}-c^{3ct}=c^{3ct}(c^{(2h-3)ct}-1)>ht.$$
But now
$$\log_c\lambda< 2^{c^{3cht}-c^{2cht}+ht}<2^{c^{3cht}-c^{3ct}}=\tau.$$
This proves \eqref{st:lambdabound}.
\medskip

For each $i\in \set{1}{c}$, let $\mathscr{L}_i=(\mca{H}_{i,1},\ldots, \mca{H}_{i,t_i})$, where $t_i\in \set{1}{t}$. We need two definitions:

\begin{itemize}[leftmargin=8mm, itemsep=1mm]
    \item Let $s_1,\ldots, s_c\in \mathbb N$ such that $s_i\leq t_i$ for every $i\in \set{1}{c}$. We say that a subset $\Omega$ of $V$ is \textit{$(s_1,\ldots, s_c)$-admissible} if for every $i\in \set{1}{c}$, there is an $s_i$-subset $J_i$ of $\set{1}{t_i}$ such that $G_i[\Omega]$ has no $(\mca{H}_{i,j}:j\in J_i)$-complete induced subgraph.
    \item Let $\nu:2^{V}\to \mathbb N\cup\{0\}$ be the function with the following rule: Let $\nu(\varnothing)=0$, and for every nonempty $\Omega\subseteq V$, let $\nu(\Omega)\in \set{1}{|\Omega|}$ be minimum such that for some $\nu(\Omega)$ subsets $\Omega_1, \ldots, \Omega_{\nu(\Omega)}$ of $V$, we have
\begin{itemize}[leftmargin=8mm, itemsep=1mm]
    \item $\Omega\subseteq\Omega_1\cup \cdots\cup \Omega_{\nu(\Omega)}$; and
    \item for each $j\in \set{1}{\nu(\Omega)}$ with $|\Omega_j|>1$, there exists $i\in \set{1}{c}$ such that $G_i[\Omega_j]$ is $\mca{H}_{i,k}$-free for some $k\in \set{1}{t_i}$.
\end{itemize}

\noindent Note that this is well-defined because $\Omega=\bigcup_{x\in \Omega}\{x\}$ is finite. 
\end{itemize}

In the above terminology, the assumption of \ref{thm:main} says that $V$ is $(t_1,\ldots, t_c)$-admissible, and the conclusion of \ref{thm:main} says that $\nu(V)<n+1$. Instead, we prove the following strengthening that is tailored to an inductive argument:

    \sta{\label{st:induction} Let $s_1,\ldots, s_c\in \mathbb N$ such that $s_i\leq t_i$ for every $i\in \set{1}{c}$, and let $\Omega$ be a nonempty subset of $V$ which is $(s_1,\ldots, s_c)$-admissible. Then $\nu(\Omega)\leq \lambda^{(s_1+\cdots+s_c)-c}$.}

    The proof is by induction on $\sigma=s_1+\cdots+s_c\geq c$. Assume that $s_i=1$ for some $i\in \{1,\ldots,c\}$. Since $\Omega$ is $(s_1,\ldots, s_c)$-admissible, it follows that there exists $j\in \set{1}{t_i}$ such that $G_i[\Omega]$ is $\mca{H}_{i,j}$-free. In this case, $\nu(\Omega)=1\leq \lambda^{\sigma-c}$. Therefore, we may assume that $s_1,\ldots, s_c\geq 2$ (and so $\sigma\geq 2c$). Suppose for a contradiction that $\nu(\Omega)>\lambda^{\sigma-c}$. Let $\alpha=\lambda^{c-\sigma}$ and let $\beta=2/\lambda$. Then $\alpha,\beta\in (0,1]$ and
    $$\alpha+\beta a=\alpha +2a/\lambda=\alpha+2c^{-cht}/3.$$
    Note also that $\alpha<1/\lambda<c^{-cht}/3$. Thus,
    $$\alpha+\beta a< c^{-cht}.$$
    Since $\Omega$ is $(s_1,\ldots, s_c)$-admissible, for every $i\in \set{1}{c}$, there is an $s_i$-subset $J_i$ of $\set{1}{t_i}$ such that $G_i[\Omega]$ has no $(\mca{H}_{i,j}:j\in J_i)$-complete induced subgraph. For each $i\in \set{1}{c}$, fix a $(\mca{H}_{i,j}:j\in J_i)$-complete graph $H_i$; then, $|V(H_i)|\leq hs_i\leq ht_i\leq ht$, and $G_i$ is $H_i$-free. Let $\mu:2^{\Omega}\to [0,1]$ be the set-function on $\Omega$ with the rule:
    $$\mu:X\mapsto \nu(X)/\nu(\Omega).$$ 
    It is easy to check that $(\Omega,\mu)$ is a space. Also, $\mu(\{x\})=\nu(\{x\})/\nu(\Omega)<1/\lambda^{\sigma-c}=\alpha$ for every $x\in \Omega$; thus, $(\Omega,\mu)$ is $\alpha$-atomic. Recall that by \eqref{st:lambdabound}, we have $b\geq \max\{c,2ht\}$, and observe that 
    $$a> b^{(2c+2)^{\tbinom{c^{cht}}{2}-1}}.$$
   So we can apply \Cref{lem:maindense} (with $\eta=1/(2ht)$) to deduce that there exist $k\in \set{1}{c}$ and disjoint $X,Y\subseteq \Omega$ such that $\mu(X),\mu(Y)\geq \beta$ and $(X,Y)$ is $(\mu,1/(2ht))$-sparse in $\ol{G_k}$. It follows that
    $$\nu(X),\nu(Y)\geq \beta \nu(\Omega)>\beta\lambda^{\sigma-c}=2\lambda^{(\sigma-1)-c}.$$
    In particular, $\nu(X)>1$ (because $\sigma>c$), which implies that $|X|>1$ and for every $j\in \set{1}{t_k}$, there exists $X_j\subseteq X$ such that $G_k[X_j]$ is isomorphic to some graph in $\mca{H}_{k,j}$. Let $$X'=X_1\cup \cdots\cup X_{t_k}$$ and let 
    $$Y'=Y\setminus \left(\bigcup_{x\in X'}N_{\ol{G_k}}(x)\right)=Y\cap \left(\bigcap_{x\in X'}N_{G_k}(x)\right).$$
    Then $|X'|\leq ht_k\leq ht$, and every vertex in $X'$ is adjacent in $G_k$ to every vertex in $Y'$. Since $(X,Y)$ is $(\mu,1/(2ht))$-sparse in $\ol{G_k}$, it follows that 
    $$\nu(Y')\geq \nu(Y)- (|X'|/(2ht))\nu(Y)\geq \nu(Y)/2> \lambda^{(\sigma-1)-c}.$$
Let $s'_k=s_{k}-1$ and let $s'_{i}=s_{i}$ for all $i\in \set{1}{c}\setminus \{k\}$. Then $s'_1+\cdots+s'_c=\sigma-1$. Since $\nu(Y')>\lambda^{(\sigma-1)-c}$, by the inductive hypothesis, $Y'$ is not $(s'_1,\ldots, s'_c)$-admissible. It follows that there exists $i\in \set{1}{c}$ such that for every $s'_{i}$-subset $J'_{i}$ of $\set{1}{t_i}$, there exists $Y_{J'_{i}}\subseteq Y'$ such that $G_i[Y_{J'_{i}}]$ is a $(\mca{H}_{i,j}:j\in J'_{i})$-complete graph. There are now two possibilities:

\begin{itemize}[leftmargin=8mm, itemsep=1mm]
    \item Assume that $i\neq k$. Then $s'_{i}=s_{i}$, and for every $s_{i}$-subset $J_{i}$ of $\set{1}{t_i}$, $G_i[Y_{J_{i}}]$ is a $(\mca{H}_{i,j}:j\in J_{i})$-complete induced subgraph of $G_i[\Omega]$. 
    \item Assume that $i=k$. Then $s'_i=s_i-1$. Let $J_i$ be an $s_{i}$-subset of $\set{1}{t_i}$, and pick $j_0\in J_i$. Then $J'_i=J_i\setminus \{j_0\}$ is an $s'_i$-subset of $\set{1}{t_i}$, and so $G_i[Y_{J'_i}]$ is a $(\mca{H}_{i,j}:j\in J'_i)$-complete graph. Recall (by $i=k$) that $G_i[X_{j_0}]$ is isomorphic to a graph in $\mca{H}_{i,j_0}$, and also every vertex in $X_{j_0}\subseteq X'$ is adjacent in $G_i$ to every vertex in $Y_{J'_i}\subseteq Y'$. So, $G_i[X_{j_0}\cup Y_{J'_i}]$ is a $(\mca{H}_{i,j}:j\in J_i)$-complete induced subgraph of $G_i[\Omega]$.
\end{itemize}
 In either case, for every $s_{i}$-subset $J_{i}$ of $\set{1}{t_i}$, we find an induced subgraph of $G_i[\Omega]$ which is a $(\mca{H}_{i,j}:j\in J_{i})$-complete graph, a contradiction to the assumption that $\Omega$ is $(s_1,\ldots, s_c)$-admissible. This proves \eqref{st:induction}.
\medskip

Since $V$ is $(t_1,\ldots, t_c)$-admissible, by \eqref{st:induction}, we have $\nu(V)\leq \lambda^{t_1+\cdots+t_c-c}\leq \lambda^{c(t-1)}$, and so by \eqref{st:lambdabound}, we have $\nu(V)<c^{\tau c(t-1)}=n+1$. This completes the proof of \Cref{thm:main}.
\end{proof}

\bibliographystyle{abbrv}
\bibliography{ref}
\newpage

\section*{Appendix: Proof of \Cref{thm:better-hit-vs-anti}}
\phantomsection
\label{appendix}

For completeness, let us first give a proof of the explicit bound in \Cref{thm:multiramsey}:

\setcounter{theorem}{0}
\setcounter{section}{1}
\begin{theorem}[Ramsey \cite{multiramsey}]
For all $c,t\in \mathbb N$ with $c\geq 2$ and every integer $n\geq c^{c(t-1)}$, every $c$-coloring of $E(K_n)$ contains a monochromatic copy of $K_t$.
\end{theorem}

\begin{proof}[Proof of \Cref{thm:multiramsey}]
First, we show that:

\sta{\label{st:backedgecolor} For all $p,q\in \mathbb N$ with $p\geq c^{q-1}$ and every $c$-coloring of $E(K_p)$, there are $q$ vertices $x_1,\ldots,x_q\in V(K_p)$ such that for each $j\in \{1,\ldots,q\}$, the edges $(x_ix_j:i\in \{1,\ldots,j-1\})$ are colored the same.}

The proof is by induction on $q$, and the case $q=1$ is clear. Assume that $q\geq 2$. Let $v_1,\ldots,v_p$ be an enumeration of the vertices of $K_p$. Since $c\geq 2$, it follows that $p-1\geq c^{q-1}-1>c(c^{q-2}-1)$, and so there is a $c^{q-2}$-subset $X$ of $\{v_1,\ldots,v_{p-1}\}$ such that the edges $(xv_p:x\in X)$ are all colored the same. Also, by the inductive hypothesis, there are $q-1$ vertices $x_1,\ldots,x_{q-1}\in X$ such that for each $j\in \{1,\ldots,q-1\}$, the edges $(x_ix_j:i\in \{1,\ldots,j-1\})$ are colored the same. Now $x_1,\ldots,x_{q-1},v_p$ are the desired $q$ vertices of $K_p$. This proves \eqref{st:backedgecolor}.
\medskip

Let $q=c(t-1)+1$. By \eqref{st:backedgecolor}, for every $n\geq c^{c(t-1)}=c^{q-1}$ and every $c$-coloring of $E(K_n)$, there are $q$ vertices $x_1,\ldots,x_q\in V(K_n)$ such that for every $j\in \{1,\ldots,q\}$, the edges $(x_ix_j:i\in \{1,\ldots,j-1\})$ are all colored with the same color $\ell_{x_j}$. Since $q>c(t-1)$, it follows that there is a $t$-subset $T$ of $\{x_1,\ldots,x_q\}$ as well as a color $\ell$ such that $\ell_x=\ell$ for all $x\in T$. But now all edges of $K_T$ are colored with the same color $\ell$, as desired.
\end{proof}

We now prove \Cref{thm:better-hit-vs-anti}:

\betteranti*
\begin{proof}
   Assume that $\mca{S}$ has no hitting set $X\subseteq V(G)$ with $\alpha(G[X])\leq bcn$. For each $i\in \{1,\ldots,c-1\}$, let $$\mca{S}_i=\{S\in \mca{S}: |S|=i\}.$$
   Then $\mca{S}=\mca{S}_1\cup\cdots\cup\mca{S}_{c-1}$, and therefore:
   
   \sta{\label{st:nonempty} There exists $k\in \{1,\ldots,c-1\}$ such that $\mca{S}_k$ has no hitting set $X\subseteq V(G)$ with $\alpha(G[X])\leq bn$.}
   
   From here on, let $k\in \{1,\ldots,c-1\}$ be as given by \eqref{st:nonempty}. Note that $\mca{S}_k\neq \varnothing$, as otherwise $X=\varnothing$ would be a hitting set for $\mca{S}_k$. For each $S\in \mca{S}_k$, fix an enumeration $v_{1,S},\ldots,v_{k,S}$ of the vertices in $S$. For each $i\in \{1,\ldots,k\}$, let $G_i$ be the graph with $V(G_i)=\mca{S}_k$ such that distinct sets $S,T\in \mca{S}_k$ are adjacent in $G_i$ if and only if either $v_{i,S}=v_{i,T}$, or $v_{i,S}$ and $v_{i,T}$ are adjacent in $G$. For every $i\in \{k+1,\ldots,c-1\}$, let $G_i$ be the graph with $V(G_i)=\mca{S}_k$ and $E(G_i)=\varnothing$ (note that there may be no such $i$). We show that:

\sta{\label{st:bipfree} $G_i$ is $K_{b,b}$-free for every $i\in \{1,\ldots,c-1\}$.}

This is immediate for $i>k$. Suppose that for some $i\in \{1,\ldots,k\}$, there are two disjoint $b$-subsets $\mca{B}_1,\mca{B}_2$ of $\mca{S}_k$ such that the sets in $\mca{B}_j$ are pairwise nonadjacent in $G_i$ for every $j\in \{1,2\}$, and every $B_1\in \mca{B}_1$ is adjacent in $G_i$ to every $B_2\in \mca{B}_2$. Since $G$ is $K_{b,b}$-free, there exist $(B_1,B_2)\in \mca{B}_1\times \mca{B}_2$ for which $v_{i,B_1}=v_{i,B_2}$. Since $b\geq 2$, we can choose $B_3\in \mca{B}_2\setminus \{B_2\}$. It follows that $v_{i,B_2}$ and $v_{i,B_3}$ are distinct and nonadjacent in $G$ because $B_2$ and $B_3$ are distinct and nonadjacent in $G_i$. At the same time, $v_{i,B_2}$ and $v_{i,B_3}$ are either the same or adjacent in $G$ because $v_{i,B_1}=v_{i,B_2}$ and $B_1,B_3$ are adjacent in $G_i$, a contradiction. This proves \eqref{st:bipfree}.

\sta{\label{st:commonanti} There is a $t$-subset $\mca{T}$ of $\mca{S}_k$ such that for all $i\in \{1,\ldots,k\}$, the vertices $(v_{i,T}:T\in \mca{T})$ are pairwise distinct and nonadjacent in $G$.}

Let $G_c$ be the graph with $V(G_c)=\mca{S}_k$ and
$$E(G_c)=E(K_{\mca{S}_k})\setminus (E(G_1)\cup\cdots\cup E(G_{c-1}))=E(K_{\mca{S}_k})\setminus (E(G_1)\cup\cdots\cup E(G_k)).$$
It suffices to show that $G_c$ contains a copy of $K_t$. Suppose not. Note that $(G_1,\ldots,G_c)$ is a $c$-multicoloring of $E(K_{\mca{S}_k})$. By \eqref{st:bipfree}, $G_1,\ldots,G_{c-1}$ are all $K_{b,b}$-free; note that $K_{b,b}$ is the complete join of $2<t$ copies of $\ol{K_{b}}$. Also, $G_c$ is $K_t$-free, and note that $K_t$ is the complete join of $t$ copies of $K_1$. Therefore, we can apply \Cref{thm:main} to obtain a cover of $\mca{S}_k$ by $n$ sets $\mca{T}_1,\ldots,\mca{T}_n\subseteq \mca{S}_k$ such that for every $j\in \{1,\ldots,n\}$, one of the following holds for some $i_j\in \{1,\ldots,c\}$:
\begin{itemize}[leftmargin=8mm, itemsep=1mm]
\item $G_{i_j}[\mca{T}_j]$ is isomorphic to $K_1$ (this is the case where $|\mca{T}_j|=1$, and $i_j\in \set{1}{c}$ can be chosen freely).
    \item $i_j\in \{1,\ldots,k\}$ and $G_{i_j}[\mca{T}_j]$ is $\ol{K_{b}}$-free.
    \item $i_j\in \{k+1,\ldots,c-1\}$ and $G_{i_j}[\mca{T}_j]$ is $\ol{K_{b}}$-free.
    \item $i_j=c$ and $G_{i_j}[\mca{T}_j]$ is $K_1$-free; that is, $\mca{T}_j=\varnothing$.
\end{itemize}
For each $j\in \{1,\ldots,n\}$, let $X_j=\{v_{i_j,T}:T\in \mca{T}_j\}$. If the first outcome holds for $j$, then $|X_j|\leq |\mca{T}_j|=1$. If the second outcome holds, then $G[X_j]$ is $\ol{K_{b}}$-free. If the third outcome holds, then $|X_j|\leq |\mca{T}_j|\leq b$ (because $E(G_{i_j})=\varnothing$). If the fourth outcome holds, then $\mca{T}_j=\varnothing$, and so $X_j=\varnothing$. We deduce that $\alpha(G[X_j])\leq b$ for every $j\in \{1,\ldots,n\}$. Let $X=X_1\cup \cdots \cup X_n$. Then $\alpha(G[X])\leq bn$. Moreover, since $\mca{S}_k=\mca{T}_1\cup\cdots\cup \mca{T}_n$, for each $S\in \mca{S}_k$, there exists $j\in \set{1}{n}$ such that $S\in \mca{T}_j$, and so $v_{i_j,S}\in S\cap X_j\subseteq X$. But then $X$ is a hitting set of $\mca{S}_k$ with $\alpha(G[X])\leq bn$, a contradiction to \eqref{st:nonempty}. This proves \eqref{st:commonanti}.
\medskip

Let $\mca{T}$ be as given by \eqref{st:commonanti}. For every $e\in E(K_{\mca{T}})$ with ends $T_1,T_2\in \mca{T}$, let
$$I_e=\{(i_1,i_2): i_1,i_2\in \{1,\ldots,k\},\ v_{i_1,T_1}=v_{i_2,T_2}\};$$
let
$$J_e=\{(i_1,i_2): i_1,i_2\in \{1,\ldots,k\},\ v_{i_1,T_1}v_{i_2,T_2}\in E(G)\};$$
and color $e$ with the pair $(I_e,J_e)$. This yields a coloring of $E(K_{\mca{T}})$ with $2^{2k^2}<2^{2c^2}$ colors. Since $$|\mca{T}|=t=\big(2^{2c^2}\big)^{2^{2c^2}(2b+d-1)}>\big(2^{2k^2}\big)^{2^{2k^2}(2b+d-1)};$$
it follows from \Cref{thm:multiramsey} that there is a $(2b+d)$-subset $\mca{U}$ of $\mca{T}$ as well as subsets $I,J$ of $\{1,\ldots,k\}\times \{1,\ldots,k\}$ such that for every $e\in E(K_{\mca{U}})$, we have 
$$(I_e,J_e)=(I,J).$$

We further claim that:

\sta{\label{st:empty} $I=J=\varnothing$.}

Assume that $(i_1,i_2)\in I$. Choose three distinct sets $T_1,T_2,T_3\in \mca{U}$; this is possible because $|\mca{U}|=2b+d\geq 3$. Since $(i_1,i_2)\in I$, we have $v_{i_1,T_1}=v_{i_2,T_3}$ and $v_{i_1,T_2}=v_{i_2,T_3}$. But now $v_{i_1,T_1}=v_{i_1,T_2}$, a contradiction to \eqref{st:commonanti}. Assume that $(i_1,i_2)\in J$. Since $|\mca{U}|>2b$, we may choose two disjoint $b$-subsets $\mca{B}_1,\mca{B}_2$ of $\mca{U}$. By \eqref{st:commonanti}, for each $j\in\{1,2\}$, the vertices $(v_{i_j,B}:B\in \mca{B}_j)$ are pairwise distinct and nonadjacent in $G$. Also, for each $(B_1,B_2)\in \mca{B}_1\times \mca{B}_2$, since $(i_1,i_2)\in J$, it follows that $v_{i_1,B_1}v_{i_2,B_2}\in E(G)$, and in particular $v_{i_1,B_1}\neq v_{i_2,B_2}$. But then $G[\{v_{i_1,B_1}:B_1\in \mca{B}_1\}\cup \{v_{i_2,B_2}:B_2\in \mca{B}_2\}]$ is isomorphic to $K_{b,b}$, a contradiction. This proves \eqref{st:empty}.
\medskip

By \eqref{st:empty}, the sets in $\mca{U}$ are pairwise anticomplete in $G$. Recall also that $|\mca{U}|>d$. This completes the proof of \Cref{thm:better-hit-vs-anti}.
\end{proof}
\end{document}